\documentclass[12pt,oneside]{amsart}
\usepackage[pdftex]{epsfig}

\usepackage{amsmath}
\usepackage{amssymb,amsbsy,amsmath,amsfonts,amssymb,amscd}
\usepackage{latexsym}
\usepackage{geometry}
\usepackage{showkeys}
\usepackage{color, colortbl}

\usepackage{latexsym} 
\usepackage{comment}

\usepackage{color}
\usepackage{colonequals}

\definecolor{darkgreen}{rgb}{0,0.5,0}

\usepackage[
        colorlinks, citecolor=darkgreen,
        pdfauthor={Ingrid Bauer, Fabrizio Catanese, Antonio Di Scala},
        pdftitle={Higher dimensional lemniscates}
]{hyperref}

\usepackage[all]{xy} 

\usepackage[alphabetic,backrefs,lite]{amsrefs}

\input xy
\xyoption{all}

\newcommand\sC{{\mathcal C}}

\newcommand\sG{{\mathcal G}}

\newcommand\sL{{\mathcal L}}

\newcommand\sR{{\mathcal R}}

\newcommand\Ga{\Gamma}
\newcommand\De{\Delta}

\newcommand\de{\delta}

\newcommand{\CC}{\ensuremath{\mathbb{C}}}
\newcommand{\RR}{\ensuremath{\mathbb{R}}}

\newcommand{\ra}{\ensuremath{\rightarrow}}

\def\eea{\end{eqnarray*}}
\def\bea{\begin{eqnarray*}}

\newcommand\dual{\mathrel{\raise3pt\hbox{$\underline{\mathrm{\thinspace d
\thinspace}}$}}}
\newcommand\qe{\ifhmode\unskip\nobreak\fi\quad $\Box$}       

\def\BOX{\hfill\lower.5\baselineskip\hbox{$\Box$}}

\newtheorem{theo}{Theorem}[section]
\newtheorem{remarkk}[theo]{Remark}
\newenvironment{rem}{\begin{remarkk}\rm}{\end{remarkk}}

\newtheorem{defin}[theo]{Definition}
\newenvironment{definition}{\begin{defin}\rm}{\end{defin}}

\newtheorem{prop}[theo] {Proposition}
\newtheorem{cor}[theo]{Corollary}
\newtheorem{lemma}[theo]{Lemma}
\newtheorem{example}[theo]{Example}

\newtheorem{question}[theo]{Question}

\DeclareMathOperator{\grad}{grad}
\DeclareMathOperator{\Mat}{Mat}
\DeclareMathOperator{\Conv}{Conv}
\begin{document}

\title[Space Lemniscates]{Higher dimensional Lemniscates: the geometry of $r$ particles in $n$-space with logarithmic potentials}
\author{I. Bauer, F. Catanese, A.J. Di Scala}

\thanks{The authors acknowledge support of the ERC 2013 Advanced Research Grant - 340258 - TADMICAMT}

\date{\today}
\maketitle

\begin{abstract}

The main purpose of this paper is to lay down the foundations of the theory of higher dimensional lemniscates
 and to propose the investigation of several interesting open problems in the field.

 The underlying geometrical problem is quite old (see \cite{walsh6},  \cite{walsh7}, \cite{nagy4}, \cite{motzkin-walsh}),
 but, whereas in the past the main question was the location of the critical points of a certain distance function,
 the  main thrust of our approach is to look at the logarithm of this function as the sum of  logarithmic potentials at $r$ points.
And the main goal is  the topological analysis of the configuration of the singular solutions (of the associated differential equation).
  We  look at the number and type of critical points, and
we make a  breakthrough in higher dimensions relating the question to modern methods of  complex analysis
 in several variables,  showing that any  critical point   has   Hessian with   positivity at least $(n-1)$;
 hence, for general choice of the $r$ points  we get  a local Morse function whose  only critical points are local minima and saddles of negativity 1.

  Moreover we  show that  the critical points are   isolated, so that  there are no    curves of local minima.
 The existence or not of these curves, and the upper bound for the number of critical points are related to  famous classical problems
 posed by Morse-Cairns (\cite{cairns-morse}  and by Maxwell \cite{maxwell}  for the case of electrostatic potentials.

 We provide examples where the number of extra local minima can be arbitrarily large.

\end{abstract}
\tableofcontents
\section*{Introduction}

\noindent
{\em What is a lemniscate?}

\noindent
 If the reader looks on the web, or in the classical literature, as an answer,  he will find the lemniscate of Bernoulli, which is the singular level set of the
 function
 $|z^2 - c^2|$. The  level sets of such a function are  called {\em Cassini's ovals}: they are just the sets of points $z$ in the complex plane $\CC$ such that the product
 $|z-c||z+c|$ of the respective distances of $z$  from the points $+c, -c$, where $c \in \RR$,
 is equal to a constant (whereas ellipses are defined by the condition that the sum of the distances from two given points is constant).

 Since any univariate polynomial $P(z) \in \CC [z]$ is a product of linear factors, $P(z) =  \prod_1^r (z-w_j)$, we see that if we take $r$ points $w_1, \dots, w_r$
 in the plane, the locus of points $z$ such that the product of the $r$ respective distances  $|z-w_j|$ is equal to a constant, is just a level set
 of the absolute value  $|P(z) |$ of the complex polynomial $ P(z)$.

 Using this analogy, the second author and Marco Paluszny in \cite{catanesepaluszny} defined a big lemniscate as a singular level
 set  of the absolute value  $|P(z) |$ of the complex polynomial $ P(z)$, and a small lemniscate as a singular connected component of a
 big lemniscate.  These definitions  were motivated by  the quest of  understanding the
singular solutions of certain differential equations (\cite{pal1}).

If we take the square $F(z) = |P(z) |^2 = P(z) \overline{P(z)}$, we obtain a real polynomial $F \in \RR[x,y]$ and if the points $w_1, \ldots , w_r$ are distinct,
they are absolute minima with nondegenerate Hessian, and the major results of \cite{catanesepaluszny} (relying also on the analysis in \cite{CatWaj})
consisted in describing the
topological configurations of the union of the big lemniscates (resp.: of the small lemniscates) in the special situation
where $f : = \log (F)$ is a global Morse function, i.e, a function whose critical points $y_i$ all have a non-degenerate Hessian, and such that all the critical values $v_i : = f(y_i)$ are different.

Here, the critical points of $F$ are just the absolute minima $w_1, \dots, w_r$, and the roots $y_1, \dots, y_{r-1}$ of the complex derivative $P'(z)$:
the points $y_i$ have (negativity) index 1, and are thus saddle points.

There is a beautiful order which governs the pictures of these lemniscates, and leads to nice and interesting generating functions
which seem to be ubiquitous (see \cite{Arnold91},  \cite{Arnold92} which took up the pictures of   \cite{catanesepaluszny}, and \cite{CatFr93}
which explained geometrically why the generating functions for big lemniscates and Morse functions on $\RR$ are the same)
The key idea is to enumerate the components where the topological configuration is fixed as the orbits of a subgroup of the braid group
acting on the set of edge labelled trees.

While the first two authors (\cite{bauercat}) generalized  these investigations to the case of an algebraic function on a Riemann surface (i.e.,
a connected complex manifold
of complex dimension equal to 1), Marco Paluszny (\cite{pmo} and \cite{3D}) defined and started to investigate the similar loci in $\RR^3$, producing nice pictures of the
corresponding configurations, thus bringing to attention a classical problem investigated by Walsh and others  \cite{walsh6},  \cite{walsh7}, \cite{nagy4}, \cite{motzkin-walsh}.

The first purpose of the present paper is to lay the foundations of the theory of lemniscates in $\RR^N$, proving some rather strong
basic results.

\begin{defin} \

\begin{enumerate}
\item
Let $w_1, \dots, w_r \in \RR^N$ be distinct points, and consider the functions
$$F :  \RR^N \ra \RR^+, F(x) : =  \prod_1^r |x-w_j|^2, \ \  f(x) : = \log F(x) = \sum_1^r  \log (|x-w_j|^2).$$
\item
A {\em big lemniscate} (for  $w_1, \dots, w_r \in \RR^N$) is defined to be a singular level set $\Ga_c$ of $f(x)$.
The big  lemniscate configuration of $f$ is the union $\Ga(f) $ of the singular level sets  $\Ga_c$.
\item
A {\em small lemniscate}  (for  $w_1, \dots, w_r \in \RR^N$) is defined to be a connected component $\Lambda_c$
of a level set  $\Ga_c = \{ x | f(x) = c\}$,
which is singular. The small  lemniscate configuration of $f$ is the union $\Lambda(f) $ of the small lemniscates.
\item
The configuration  $\Lambda(f) $  of small lemniscates is said to be {\em weakly generic} if the function $f(x)$  is a (local) Morse function, i.e.,
its critical points $y_i$ all have a non-degenerate Hessian.
\item
The configuration of big lemniscates $\Ga(f) $ is said to be {\em generic} if the function  $f(x)$  is a global Morse function, i.e., it is a Morse function
and the critical values $f(y_i)$ are all different (notice that the absolute minima for $F(x)$ are just the zeros of $F(x)$,
i.e.,  the points $w_j$, which are all automatically non degenerate).
\end{enumerate}
\end{defin}

\begin{figure}[h!]
\centering
\includegraphics[trim=1cm 8cm 6.5cm 10cm, clip,draft=false,width=8cm,height=7cm]{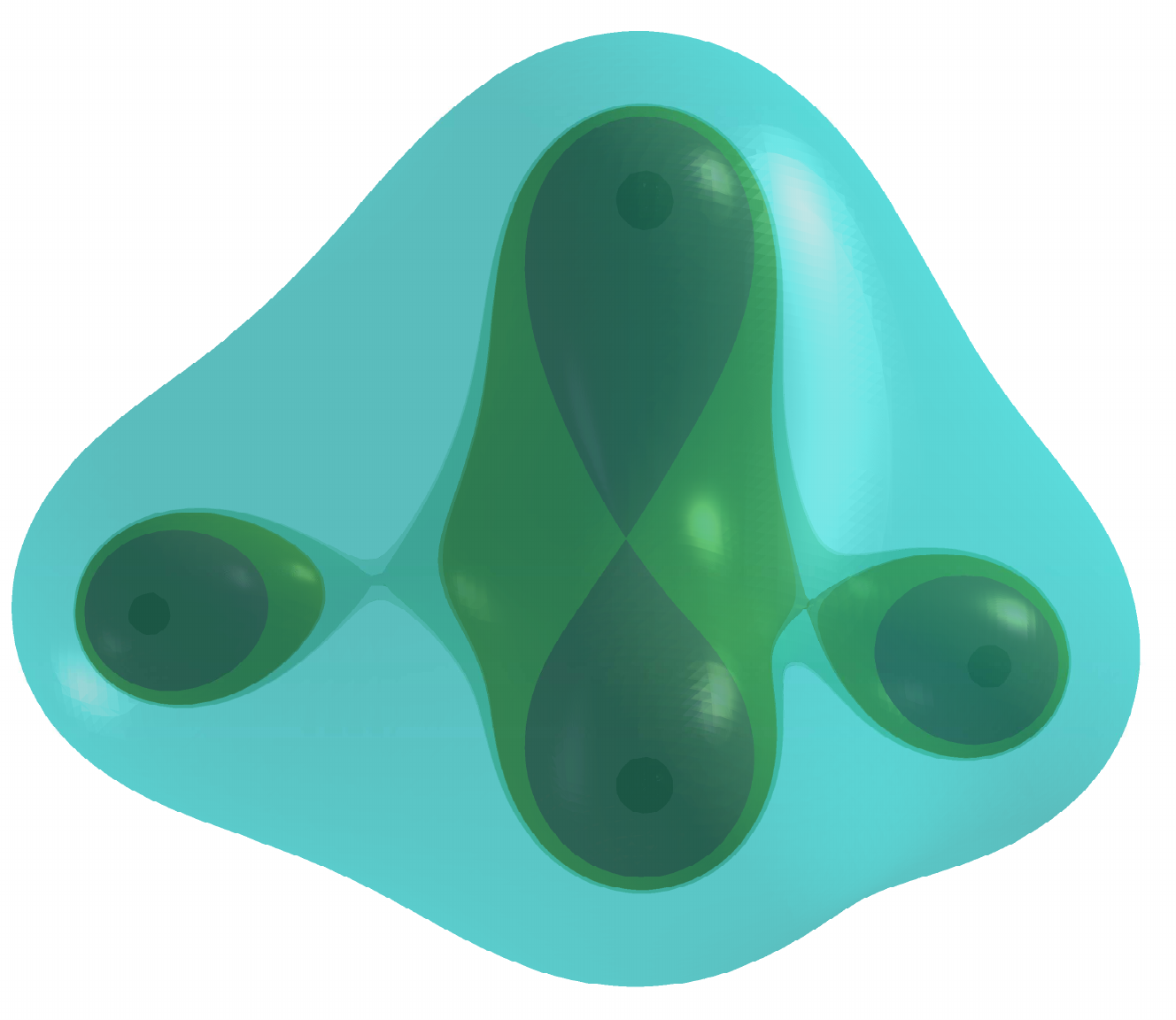}
\caption{A generic big lemniscate configuration for $r=4$ surrounded by a non singular level set. The four points are still visible at the interior of the lemniscates.}
\end{figure}

In the case where the points $w_1, \dots, w_r $ lie in an affine plane contained in $ \RR^N$, the situation is easy to analyse, see
corollary \ref{orthodec},
and, like in the case $N=2$, if $f$ is a Morse function, we have just $r-1$ critical points of (negativity) index 1.

The following are our main theorems.

\begin{theo}\label{index}
Let $w_1, \dots, w_r \in \RR^N$ be distinct points, not lying in a (real) affine plane, and consider the functions
$$F :  \RR^N \ra \RR^+, F(x) : =  \prod_1^r |x-w_j|^2, \ \  f(x) : = \log F(x) = \sum_1^r  \log (|x-w_j|^2).$$

Then at every critical point of $f(x)$ (resp.: of $F(x)$) the Hessian has positivity at least $N-1$.

\noindent
In particular, the index of negativity can only be $0$ or $1$.

\noindent
 Moreover, the critical points are isolated.  
\end{theo}

\begin{theo}\label{extra}

Let $w_1, \dots, w_r \in \RR^N$ and let
$$F :  \RR^N \ra \RR^+, F(x) : =  \prod_1^r |x-w_j|^2, \ \  f(x) : = \log F(x) = \sum_1^r  \log (|x-w_j|^2)$$
be as in theorem \ref{index}.

\begin{enumerate}
\item Assume that  $F$ is a (local) Morse function: then $F(x)$ has $r$ absolute minima, $h$ local minima, and exactly $r+h-1$
critical points of (negativity) index 1.
\item
There are examples already in $\RR^3$, where $h$ can be arbitrarily large.
\end{enumerate}
\end{theo}

\begin{rem}
Part (1) of the above theorem is a direct consequence of theorem \ref{index} and standard Morse theory. Part (2) is shown in section \ref{preassigned}, see in particular proposition \ref{hminima}.
\end{rem}

\noindent
 The results are based, once again, on elementary complex analysis: but this time in several variables.

 The first idea is to take an isometric  embedding of $\RR^N$
 into $\CC^n$ ($n = [\frac{N-1}{2}]+1$).
  This is crucial, since on a complex vector space any real bilinear form can be written as the sum
  $Q + \sL + \bar{Q}$
  where $Q$ is a complex bilinear form, and $\sL $ (the Levi form) is Hermitian:
  the easiest case being $n=1$, where, if $ z = x+iy$, $a,b,c \in \RR$,
  $$  (a+ib) z^2 +  (a-ib) \bar{z}^2 + c z \bar{z} = 2 a (x^2 -y^2) - 4b xy + c (x^2 +y^2) .$$

\noindent
 Then we prove (under the assumption that $w_1, \ldots , w_r$ are not contained in a  complex line), using the classical Fubini-Study form,
  that the function $f$ is strictly plurisubharmonic, i.e., its Levi form $\sL $  is strictly positive definite.

\noindent
  The second trick is then to choose an appropriate isometric embedding as above, in order to  prove the statement about the positivity of the Hessian at each critical point.
  
   The first statement, that 
 the critical points have positivity at least $N-1$, put together with  the generalized Morse lemma, shows that
 there are local analytic coordinates $u_1, \dots, u_N$ such that $f$ has one of these two local normal forms:
\begin{equation}
f(u) =  u_1^2 + \dots + u_{N-1}^2 \pm  u_N^k;
\end{equation}
 \begin{equation}
 f(u) =  u_1^2 + \dots + u_{N-1}^2.
 \end{equation}
  In the first case the critical points are isolated.  In the second case, we use the fact that the critical set is 
   compact.  The compactness of the set of critical points is shown by the generalization of a theorem of Gauss:
  lemma \ref{gauss} asserts  that the critical points lie in the convex hull of the points
 $w_1, \dots, w_r$.
  
  Hence in the second case if the critical points are not isolated, 
  we would have   smooth curves of minima which are diffeomorphic to circles.
  
  The occurrence of these curves is excluded in the appendix (theorem \ref{MorseCairns}) where we give two different proofs,
  one using  Morse theory and the other using Douglas'  solution of the Plateau problem.

  For theorem \ref{extra}, we just use topology and Morse theory, and we exploit symmetry in order to produce
  `extra' local (but not global) minima.

\noindent
  The topological configuration is then easily described by the graph whose edges are the saddle points of $F$,
  in a similar fashion to  in \cite{catanesepaluszny}.

  The second main purpose of this paper is to propose as a theme of investigation the description of the configurations
  of generic big and small lemniscate configurations.

\noindent
  Even if we cannot use the Riemann existence theorem as in real dimension two, our first main theorem 
  \ref{index} yields
  a very strong information.

 \begin{defin} \label{generating}\
 \begin{enumerate}
 \item Define $\sG \sL (r,N)$ as the open set of the space  $(\RR^N)^r$of $r$ (distinct) points in $\RR^N$
 such that  the function $f(x) = \sum_{k=1}^r \log |x-w_k|^2$ is a global Morse function, i.e. we have a
 generic big lemniscate configuration $\Ga(f)$.
 \item We say that two big lemniscate configurations  $\Ga(f_1)$, $\Ga(f_2)$ have {\emph the same topological type}
 if there is a homeomorphism of the pair  $(\RR^N, \Ga(f_1))$ with the pair  $(\RR^N, \Ga(f_2))$.

 \noindent
 We have then a map from the set $\pi_0( \sG \sL (r,N))$ of the connected components of the set of
 lemniscate generic r-tuples of points to the set of topological types.
 \item  Denote by $b(r,N)$ the number of connected components of $\sG \sL (r,N)$ and by $a(r,N)$ the corresponding
 number of topological types of the big lemniscate configurations, by $c(r,N)$ the corresponding
 number of topological types of the small lemniscate configurations.

\noindent
 We define the corresponding generating functions as
 $$B_N (t) : = \sum_r  \frac{b(r+2,N)}{r!} t^r , \  A_N (t) : = \sum_r   \frac{a(r+2,N)}{r!} t^r, $$
 $$ C_N (t) : = \sum_r   c(r+2,N)  t^r.$$

\noindent
Similarly, define $b(r,N,h)$ the number of connected components of $\sG \sL (r,N)$ where $f$ has $h$ local minima,
and define similarly $$a(r,N,h), \ c(r,N,h), \ B_{N,h}(t), A_{N,h}(t),  C_{N,h}(t).$$
\end{enumerate}
 \end{defin}

\begin{figure}[h!]
\centering
\includegraphics[trim=6cm 11cm 6.5cm 10cm, clip,draft=false,width=8cm,height=7cm]{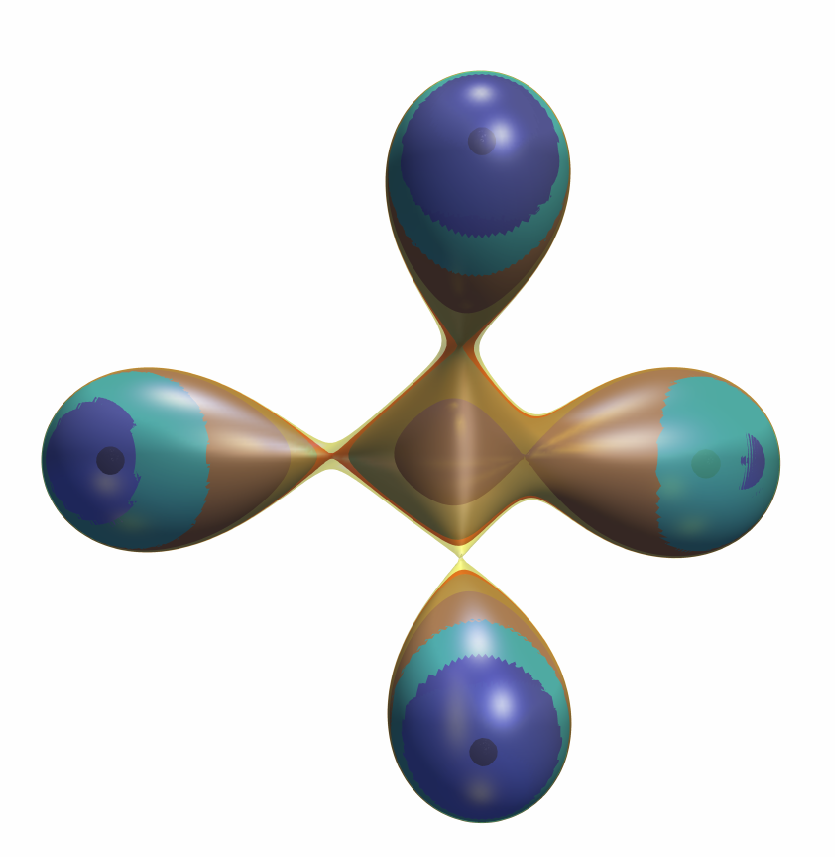}
\caption{Another generic big lemniscate configuration for $r=4$. }
\end{figure}

 In the appendix to \cite{catanesepaluszny} it is shown that, in dimension $m=2$, where $h=0$, $b(r,2) = a(r,2)$
 and the function
 $$A_2(t) = B_2(t) = \frac{1}{1 -  \sin t}.$$
 More complicated results were shown for the number of small lemniscate configurations.

 In this view, these are the questions which we would like to pose.

 \begin{question} \label{question}
 \begin{enumerate}
 \item
 Given $(r,N)$, which is the maximal number $M(r,N) = max \{h\}$ of (non global) local minima for the function $f$?
 \item
 What is the form of the generating functions $A_N (t) , B_N(t) , B_{N,h}(t), A_{N,h}(t)$?
 \item
 Is it true that the $A_N = B_N$ as in the case $N=2$?

 \end{enumerate}
 \end{question}

 To shed  light on the above questions, let us observe that the connected components of the configuration space of $r$ distinct points
 consist of several domains,
  separated by walls of two different types.

  The first type of walls contain as general points  r-tuples $w_1, \dots , w_r$ such that the associated function $f$
    has a  simple singularity of the form $u_1^2 + \dots + u_{N-1}^2+ u_N^3$;
   a pair of nondegenerate critical points, of respective indices 0, 1, go to disappear when crossing the wall in one direction: we shall call these
   {\em walls of quantitative type,} since the number of critical points changes.

   The second type of walls are those where two critical values become equal: here crossing the wall the topological type
   of the big lemniscate configuration may change, hence we shall call these {\em walls  of qualitative type.}

   Finally, concerning the first question, we get examples with $4$ points in $\RR^3$ and $h=1$, and, for every $h\geq 2$,  with $r= 3h$ points
    in $\RR^3$ and $h$ non global minima. These examples suggest the conjecture that $h$ may be bounded by a linear function  $c_1 r + c_2$.

  Let's end this introduction by describing a straightforward but potentially quite useful application of the  Gauss-type lemma \ref{gauss}, showing that the critical
  points lie in the convex hull of the points
 $w_1, \dots, w_r$.

 \begin{theo} \label{convexhull}
 1) Let $\Omega \in \RR^N$ be a bounded domain and let $w_1, \ldots , w_r \in \RR^N $ be $r$ pairwise distinct points. Consider
$$
f \colon \bar{\Omega} \rightarrow \RR \cup \{-\infty\}, \ \  f(x) = \sum_{k=1}^r \log |x-w_k|^2.
$$
Then all maxima of f (in $\bar{\Omega}$) are contained in $\partial \Omega$.

\noindent
2) Moreover, if the closure $\bar{\Omega} $  euqals the convex hull $ Conv(\{w_1, \ldots , w_r\})$ of the points $w_1, \ldots , w_r$, then all maxima of $f|\bar{\Omega}$ are contained in $\partial \Omega \setminus \mathcal{F}$, where $\mathcal{F}$ is the union of   the interior parts of the
faces of $\partial \Omega$ of dimension at least two ($\iff$ all maxima are contained in one dimensional faces of $\partial \Omega$).
 \end{theo}

 \begin{rem}
 Of course it is very  interesting from the point of view of physics to consider also the case of non logarithmic potentials,
 of the type
 $$   f(x) = \sum_{k=1}^r   |x-w_k|^{\alpha},$$
 and one can ask analogous questions to \ref{question} in the Newtonian  and in  the electrostatic case $\alpha= - 1$.
 As mentioned in the abstract, the existence or not of curves of local extremals, and the upper bound for the number of critical points are related to  famous classical problems
 posed by Morse-Cairns (\cite{cairns-morse}  and by Maxwell \cite{maxwell}  for the case of electrostatic potentials.

 We refer to \cite{shapiro} for new results in this direction.

 \noindent
 Our slogan here is that, once methods from complex analysis can be introduced, the problems become drastically simplified,
 and simple and beautiful answers can be found.
 But it is not clear  that one can use methods of complex analysis   in  more general situations, even restricting
 to open sets of parameters whose complements are of measure zero, and even in dimension n=2.

 \end{rem}

\section{The Hermitian Levi form}

\begin{definition}
Let $U \subset \CC^n$ be a domain and let $f \colon U \rightarrow \RR$ be a function, which is twice continuously differentiable. Moreover, let $z_0 \in U$ be a point. The Hermitian form  $\mathcal{L}_{f, z_0}$ given by
$$
\mathcal{L}_{f, z_0} (w):=\sum_{i,j=1}^n \frac{\partial^2 f(z_0)}{\partial z_i \partial \bar{z}_j} w_i \bar{w}_j, \ \ w = \begin{pmatrix} w_1\\ w_2\\ . \\. \\. \\w_n \end{pmatrix} \in \CC^n,
$$
is called the {\em (Hermitian) Levi form of $f$ at $z_0$}.
\end{definition}

Here we use the standard formalism: if $z=x+iy$, then
$$
\frac{\partial }{\partial z_j}  = \frac 12 (\frac{\partial }{\partial x_j} - i \frac{\partial }{\partial y_j}) \ , \ \  \frac{\partial }{\partial \bar z_j}  = \frac 12 (\frac{\partial }{\partial x_j} + i \frac{\partial }{\partial y_j}) \ .
$$

Then we have the following
result related to  the Fubini-Study metric.

\begin{lemma}\label{log}
 Let  $U \subset \CC^n \setminus \{0\}$.
Then the Hermitian Levi form  $\mathcal{L}_{f, z}$ of $f(z) : = \log|z|^2$ is positive semidefinite for each $0\neq z \in U$.

\noindent
Moreover, at each point $0\neq z \in U$, the Levi form $\mathcal{L}_{f, z}$ of $f$ at $z$ has positivity $n-1$, and the kernel is the line $\CC  z$.
\end{lemma}

\begin{proof}
We denote, for $v, w \in \CC^n$, the {\em standard Hermitian product} by
$$
\langle v,w \rangle := \sum_{k=1}^n v_k \bar w_k.
$$
Then for $z \in U$ and $w \in \CC^n$ we have
\begin{multline}
\mathcal{L}_{f,z}(w) = \sum_{i,j=1}^n \frac{\partial^2(\log(\sum_{k=1}^n z_k \bar z_k))}{\partial z_i \partial \bar z_j} w_i \bar w_j = \sum_{i,j=1}^n \frac{\partial}{\partial z_i}  ( \frac{z_j}{|z|^2} ) w_i \bar w_j = \\
= \sum_{i,j=1}^n \frac{|z|^2 \delta_{ij} - \bar z_i z_j}{|z|^4} w_i \bar w_j = \frac{1}{|z|^4} (|z|^2|w|^2 - |\langle w,z \rangle |^2).
\end{multline}
From the Cauchy-Schwartz inequality it follows now
\begin{itemize}
\item $\mathcal{L}_{f,z} \geq 0$ for all, $0 \neq z \in U$, and
\item $\mathcal{L}_{f,z}(w) = 0$ if and only if $w \in \CC z$.
\end{itemize}
\end{proof}

\noindent
We consider now the following situation: let $w_1, \ldots , w_r \in \CC^n$ be $r$ different points and consider the functions
$$F \colon \CC^n \setminus \{w_1, \ldots , w_r \} \rightarrow \RR, $$
respectively
$$f \colon \CC^n \setminus \{w_1, \ldots , w_r \} \rightarrow \RR$$
given by
\begin{itemize}
\item $F(z):= \prod_{i=1}^r |z-w_i|^2 = \prod_{i=1}^r F_i$, respectively
\item $f(z):= \sum_{i=1}^r \log |z-w_i|^2 = \sum_{i=1}^r f_i$.
\end{itemize}

Then we have the following:
\begin{lemma}\label{leviform} \
\begin{enumerate}
\item $f(z):= \sum_{i=1}^r \log |z-w_i|^2$ is plurisubharmonic, i.e., $\mathcal{L}_{f,z} \geq 0$ for all $z \neq 0$.

\item $\mathcal{L}_{f,z} > 0$ for all $z \neq 0$ except if all the points $w_1, \ldots , w_r$ are contained in an affine complex line.

\item  If $w_1, \ldots , w_r$ are contained in an affine complex line $L$, then  for all $z \in L$ the Levi form $\mathcal{L}_{f,z} $
has exactly nullity $1$, in the direction of $L$.
\end{enumerate}
\end{lemma}

\begin{proof}
If $z \neq w_k$, then the Levi form $\mathcal{L}_{f_k,z}$ of $f_k$ is $\geq 0$ in $z$ and its kernel is the line $\CC(z-w_k)$ by lemma \ref{log}. Since
$$
\mathcal{L}_{f,z} = \sum_{k=1}^r \mathcal{L}_{f_k,z},
$$
we see that
\begin{itemize}
\item $\mathcal{L}_{f,z} \geq 0$ for all $z \neq w_1, \ldots , w_r$, and
\item $v \in \ker(\mathcal{L}_{f,z}) \ \iff \ v \in \CC (z-w_k) \ \forall \ k=1, \ldots , r.$
\end{itemize}
Therefore $\mathcal{L}_{f,z}$ has non trivial kernel if and only if all the vectors $z-w_1, \ldots , z-w_r$ are proportional. This holds if and only if there is an element $u \in \CC^n$ such that $w_1, \ldots , w_r \in z + \CC u$.

3) Moreover, if  $v \in \ker(\mathcal{L}_{f,z}) \setminus\{0\}$, $v \in \CC u$, i.e., $v$ lies in the direction of $L$.

\end{proof}

\section{Linear complex structures}
Consider now $X = \RR^{2n}$ together with an $\RR$-valued symmetric bilinear form
$$
H \colon \RR^{2n} \times \RR^{2n} \rightarrow \RR.
$$
\begin{rem}
A {\em complex structure} (or a {\em $\CC$-structure}) on $X$ is a Hodge decomposition
$$
X \otimes _{\RR} \CC = V \oplus \bar{V},
$$
where $V \subset X \otimes _{\RR} \CC$ is a complex sub(-vector-)space, such that
$$
X = \{ v+ \bar{v} | v \in V \}.
$$
\end{rem}
We can extend $H$ to a symmetric $\CC$-bilinear form on $X \otimes _{\RR} \CC$, which we denote by $H_{\CC}$.  I.e.,  for $x=v+\bar{v}, \  y= w+\bar{w} \in X$ we have: $H(x,y) = H_{\CC}(v+\bar{v}, w+ \bar{w})$.

More precisely, we get:
\begin{equation}\label{decomp}
H(x,y) = H_{\CC}(v+\bar{v}, w+ \bar{w}) = H_{\CC}(v, w) + H_{\CC}(\bar{v}, \bar{w}) + H_{\CC}(v,  \bar{w}) + H_{\CC}(\bar{v}, w).
\end{equation}

Using the above formula, we have the following well known
\begin{lemma} \label{decbil}
Let $H$ be a symmetric (real) bilinear form on $\RR^{2n}$, endowed with a given complex structure (i.e., $\RR^{2n} \cong \CC^n$). Then there is a decomposition
$$
H = Q + \bar{Q} + \mathcal{L},
$$
where $Q$ is a symmetric $\CC$-bilinear form and $\mathcal{L}$ is a Hermitian form.
\end{lemma}

\begin{proof}
This follows from equation (\ref{decomp}) setting $Q := H_{\CC}(v,w)$ and  $\mathcal{L}:= H_{\CC}(v,  \bar{w}) + H_{\CC}(\bar{v}, w)$. Then $\bar{Q}(v,w) = \overline{H_{\CC}(v,w)} = H_{\CC}(\bar{v}, \bar{w})$.
\end{proof}

\begin{rem}\label{levi}
Let $U \subset \CC^n$ be a domain and let $g \colon U \rightarrow \RR$ be a function, which is twice continuously differentiable.  Applying lemma \ref{decbil} to the Hessian $H_g$ of $g$ in $z \in U$, we get:
$$
H_{g,z} = Q_{g,z} + \bar{Q}_{g,z} + 2 \mathcal{L}_{g,z},
$$

where the matrix of $Q_{g,z}$ is given by $( \frac{\partial ^2 g}{\partial z_j \partial z_k})_{1 \leq j,k \leq n}$ and $\mathcal{L}_{g,z}$ is the Levi form of $g$ in $z$.
\end{rem}

\begin{rem}
If the points $w_1, \ldots , w_r$ are contained in a complex line, then from the formula above (remark \ref{levi}) it follows that the non degenerate critical points of $f(z):= \sum_{i=1}^r \log |z-w_i|^2$ have negativity 1 (cf. also lemma 1.1. in \cite{catanesepaluszny}).
\end{rem}

\section{The case where the real affine span of $w_1, \ldots , w_r$ has dimension $\geq 3$}
We can prove the following
\begin{prop}\label{positivity}
Consider $X = \RR^{2n}$ with the standard Euclidean metric and let $H$ be a symmetric bilinear form with positivity $p \leq 2n-2$. Then there is a
$\CC$-structure on $X$ such that
\begin{itemize}
\item  there is a $\CC$-basis $v_1, \ldots, v_n$ (for this $\CC$-structure),
\item $\{v_1, \bar{v}_1, \ldots , v_n, \bar{v}_n \}$ is a unitary basis for the standard Hermitian product on $X \otimes_{\RR} \CC \cong \CC^{2n}$,
\item the Levi form $H_{\CC}(v ,\bar{v})$ is not positive definite.
\end{itemize}
\end{prop}

\begin{proof}
Since the positivity $p$ of $H$ fulfills $p \leq 2n-2$, there is an orthonormal basis $e_1, \ldots , e_{2n}$ (w.r.t. the Euclidean metric) such that
\begin{itemize}
\item $H(e_i,e_j) = 0$ for $i \neq j$,
\item if $\lambda_j:= H(e_j,e_j)$, then $\lambda_1, \lambda_2 \leq 0$.
\end{itemize}
Set
\begin{itemize}
\item $\hat{v}_j:= e_{2j-1} + i e_{2j}$ for $1 \leq j \leq n$, and
\item $v_j:= \frac{1}{\sqrt{2}} \hat{v}_j$, $1 \leq j \leq n$.
\end{itemize}
Then $v_1, \bar{v}_1, \ldots , v_n, \bar{v}_n$ is a unitary basis of $\CC^{2n} = X \otimes_{\RR} \CC$ endowed with the standard Hermitian product. Moreover, we have
\begin{equation}\label{negativity}
H_{\CC}(v_j,\bar{v}_j) = \frac 12 H_{\CC}(e_{2j-1} + i e_{2j}, e_{2j-1} - i e_{2j}) = \frac 12 (\lambda _{2j-1} + \lambda_{2j}).
\end{equation}
In particular $H_{\CC}(v_1,\bar{v}_1) \leq 0$.
\end{proof}

As a consequence of the above considerations we get the following:

\begin{prop}\label{negativity}
Assume that the real affine span of the points $w_1, \ldots , w_r \in \RR^{2n}$ has (real) dimension $\geq 3$. Then at each critical point $x$ of
$$
f(z):= \sum_{i=1}^r \log |z-w_i|^2 = \sum_{i=1}^n f_i
$$
the index of positivity  is at least $2n-1$.
\end{prop}

\begin{proof}
Consider the Hessian $H_f$ of $f$ at the critical point $x \in \RR^{2n}$ and assume that for the positivity $p$ it holds $p \leq 2n-2$. Then by proposition \ref{positivity} there is a complex structure on $X$ and a $\CC$-basis $v_1, \ldots, v_n$ such that  $\{v_1, \bar{v}_1, \ldots , v_n, \bar{v}_n \}$ is a unitary basis for the standard Hermitian product on $X \otimes_{\RR} \CC \cong \CC^{2n}$ and the Hermitian form $H_{f,{\CC}}(v ,\bar{v})$ in $x$ is not positive definite. But $H_{f,{\CC}}(v ,\bar{v}) = 2\mathcal{L}_{f,x}(v ,\bar{v})$, contradicting lemma \ref{leviform}.
\end{proof}
\begin{rem}
In particular, if $f:= \sum_{i=1}^r \log |z-w_i|^2$ is a local Morse function, it has only $h$ local minima in $\RR^{2n} \setminus \{w_1, \ldots , w_r \}$
and exactly $(h+r-1)$ other critical points, each with positivity $(2n-1)$ and negativity $1$.
\end{rem}

\begin{rem}
If we assume  that the r points $w_1, \ldots , w_r \in \RR^{2n}$ are contained in a real affine plane,
then without loss of generality we may assume:
\begin{itemize}
\item $X = \CC \times \CC^{n-1}$,
\item $w_1, \ldots , w_r \in \CC \times \{ 0\}$, $w_j = (\xi_j, 0)$, $\xi_j \in \CC$.
\end{itemize}
Then
$$
f(z):= \sum_{i=1}^r \log |z-w_i|^2 = \sum_{k=1}^n f_k,
$$
where $f_k (z) = f_k(z_1, z_2)= \log ( |z_1-\xi_k|^2 + |z_2|^2)$, $(z_1, z_2) \in  \CC \times \CC^{n-1}$.
\end{rem}

\section{Geometric properties of $f(x):= \sum_{i=1}^r \log |x-w_i|^2$}
We have the following extension of a classical result due to Gauss.

\begin{lemma}\label{gauss}
 Let $w_1, \ldots , w_r$ be $r$ different points in $\RR^N$. Then the critical points of
 $$
 f(x) = \sum_{k=1}^r \log |x-w_k|^2 = \sum_{k=1}^r f_k
 $$
 lie in the convex hull of $w_1, \ldots , w_r$.
\end{lemma}
\begin{proof}
Let $x \in \RR^N \setminus \{w_1, \ldots, w_r \}$ and set $g_k(x):= |x-w_k|^2$. Then $f_k = \log g_k$ and
$$
\grad(f_k)(x) = \frac{2(x-w_k)}{|x-w_k|^2} = 2 (x_k -w_k) g_k^{-1}(x).
$$
Therefore $x$ is a critical point of $f$ if an only if
\begin{multline}
\sum_{k=1}^r (x-w_k)g_k^{-1}(x) =0 \iff x\sum_{k=1}^r g_k^{-1}(x) = \sum_{k=1}^r w_k g_k^{-1}(x) \iff \\
\iff x= \frac{\sum_{k=1}^r w_k g_k^{-1}(x)}{\sum_{k=1}^r g_k^{-1}(x)} = \sum_{k=1}^r (\frac{g_k^{-1}(x)}{\sum_{k=1}^r g_k^{-1}(x)})w_k
\end{multline}
Note that for $x \in \RR^N \setminus \{w_1, \ldots, w_r \}$ we have $g_k(x) > 0$, hence also
$$
t_k := \frac{g_k^{-1}(x)}{\sum_{k=1}^r g_k^{-1}(x)} >0,
$$
and $\sum_{k=1}^r t_k = 1$. This proves the claim.
\end{proof}

\begin{lemma}\label{HesseMatrix}
Let $w_1, \ldots , w_r$ be $r$ different points in $\RR^N$. The Hesse matrix of
$f(x) = \sum_{k=1}^r \log |x-w_k|^2$ in $x$ is given by
$$
H_{f,x} = 2 \sum_{k=1}^r \frac{|x-w_k|^2E_N - 2 (x-w_k) (x-w_k)^T}{|x-w_k|^4},
$$
where $E_N$ is the $N\times N$ identity matrix and $(x-w_k)^T = (x_1-w_{k,1}, \ldots , x_N - w_{k,N})$.
\end{lemma}
\begin{proof}
In fact,
\begin{multline}
 (H_{f,x})_{ij} = \sum_{k=1}^r \frac{\partial^2(\log|x-w_k|^2))}{\partial x_i \partial x_j} = \sum_{k=1}^r \frac{\partial}{\partial x_i}  (\frac{2(x_j-w_{k,j})}{|x-w_k|^2} ) = \\
= 2 \sum_{k=1}^r \frac{|x-w_k|^2 \delta_{ij} - 2(x_j-w_{k,j})(x_i -w_{k,i})}{|x-w_k|^4}.
\end{multline}
\end{proof}

\begin{cor}\label{orthodec}
Assume that the points $w_1, \ldots , w_r $ lie in an affine subspace $V$. Without loss of generality
we may assume that we have a decomposition of $\RR^N$ in an orthogonal direct sum, i.e., $\RR^N = V \oplus V^{\bot}$ where $w_1, \ldots , w_r \in V$.

Then the critical points $x$ of $f$ lie in $V$ and the Hessian of $f$ is of the form
$H_{f} = H_{f|V} + H'$, where $H_{f|V}$ is the Hessian of $f|V$, and $H'$ is a positive definite quadratic form on $V^{\bot}$.
\end{cor}

\begin{proof}
Since the critical points of $f$ lie in the convex hull of the points $w_1, \ldots , w_r \in V$, it follows that they also lie in $V$.
It is easy to see that for the statement about the Hessian it suffices to prove the analogous statement for each summand $f_k$ of $f$. We have that the Hessian of $f_k$ equals to
$$
H_{f_k,x} = 2 ( \frac{|x-w_k|^2E_N - 2 (x-w_k) (x-w_k)^T}{|x-w_k|^4}).
$$
If $x$ is a critical point of $f$, then $x \in V$, i.e., $x = v + 0 \in V \oplus V^{\bot}$. Then (since also $w_k \in V$, the negative part of the Hessian of $f_k$ is zero on $V^{\bot}$)
$$
H_{f_k,x} (u_1, u_2) = H_{f|V}(u_1) + 2 |x-w_k|^2 (|u_2|^2).
$$
\end{proof}

The above considerations as well as proposition \ref{negativity} allow us now to prove the following result
\begin{theo}
Let $w_1, \ldots , w_r \in \RR^N$ be $r$ distinct points and let
$$
f \colon \RR^N \setminus \{w_1, \ldots , w_r\} \rightarrow \RR
$$ be given by
$$
 f(x) = \sum_{k=1}^r \log |z-w_k|^2.
$$
Then $f$ has only critical points of negativity 0 or 1.
\end{theo}

\begin{proof}
If $N=2n$ and the real affine span of $w_1, \ldots , w_r$ has dimension $\geq 3$ this follows from proposition \ref{negativity}. If $N= 2n-1$ and the real affine span of $w_1, \ldots , w_r$ has dimension $\geq 3$ then, we embed $\RR^N$ in $\RR^{2n}= \RR^N \times \RR$ and the claim follows from proposition \ref{negativity} and corollary \ref{orthodec}. If instead the  real affine span of $w_1, \ldots , w_r$ has dimension $\leq 2$, then we use lemma 1.1. of \cite{catanesepaluszny} and corollary \ref{orthodec}.

\end{proof}

\begin{proof}[Proof of Theorem \ref{index}] The first part follows from the above theorem.

To prove the second part we need the following generalized Morse Lemma (cf. e.g. Satz 5.5 in \cite{complexmanifolds}):

\begin{lemma}\label{localform} Let $\Omega \subset \RR^N$ be an open set and let $f : \Omega \to \mathbb{R}$ be a real analytic function.
Let $p \in \Omega$ be a critical point of $f$. Assume that  the Hessian $H_{f,p}$ has positivity index  $\geq N-1$.
Then there are local coordinates $u_1,\cdots, u_N$ centered at $p$ such that:
\[ f(u_1,\cdots,u_{N}) = u_1^2 + \cdots + u_{N-1}^2 + F(u_N) \, ,\]
where either $F(u_N)  = c  \pm u_N^d, \ d \geq 2$, or $ F(u_N) \equiv c, c \in \RR$.

\end{lemma}

\noindent
Denote by $\mathcal{C}$ the set of critical points of $f$ and let $p \in \mathcal{C}$.
Since the positivity of $H_{f,p}$ is at least $N-1$, we have to consider four cases for $p$:

\begin{itemize}
\item
either $p$ is a non-degenerate critical point with negativity $1$ or
\item
$p$ is a non-degenerate local minimum or
\item
 $H_{f,p}$ is  positive semi-definite, with positivity $N-1$, and we can apply the above generalized Morse lemma since $F$ is real analytic.

 Since the critical points near $p$ are solutions of the system of equations
\[ \begin{cases} u_1 = 0 \, ,\\
\vdots \, \\
u_{N-1} = 0 \, ,\\
F'(u_{N}) = 0 \,
\end{cases}\]
 $p$ is an isolated critical point if $ F'(u_N) $ is not identically zero, and a local minimum iff $d$ is even and the sign equals $+1$;
 \item
if $ F'(u_N) \equiv 0$
 the set $\mathcal{C}$ is near $p$ a $1$-dimensional embedded submanifold defined by the equations $u_1 = \cdots = u_{N-1} = 0$.
Moreover, it is clear that the points of this $1$-dimensional submanifold are local minima.
Hence a connected component of $\mathcal{C}$ is either an isolated point or a $1$-dimensional embedded submanifold consisting of local minima.
Since $\mathcal{C}$ is compact (because it is closed and bounded because contained in the convex hull of the points $w_1, \cdots, w_r$)
 the $1$-dimensional connected components of $\mathcal{C}$ are embedded circles.
 \item
  Such curves of minima do not exist, by theorem \ref{MorseCairns}.
 \end{itemize}
\end{proof}

The following result, which proves a conjecture in computer vision, is now quite obvious.

\begin{cor} (= Theorem \ref{convexhull}). \\
1) Let $\Omega \in \RR^N$ be a bounded domain and let $w_1, \ldots , w_r \in \RR^N \setminus \partial{\Omega}$ be $r$ different points. Consider
$$
f \colon \bar{\Omega} \rightarrow \RR \cup \{-\infty\}, \ \  f(x) = \sum_{k=1}^r \log |x-w_k|^2.
$$
Then all maxima of f (in $\bar{\Omega}$) are contained in $\partial \Omega$.

\noindent
2) Moreover, if $\bar{\Omega} = \Conv(\{w_1, \ldots , w_r\})$ is the convex hull of the points $w_1, \ldots , w_r$, then all maxima of $f|\bar{\Omega}$ are contained in $\partial \Omega \setminus \mathcal{F}$, where $\mathcal{F}$ is the interior of the union of the faces of $\partial \Omega$ of dimension at least two ($\iff$ all maxima are contained in one dimensional faces of $\partial \Omega$).
\end{cor}

\begin{proof} Without lost of generality we can assume $N$ to be even. The proof of 1) follows from the fact that $f$ is a plurisubharmonic function w.r.t. any complex structure on $\RR^N$ compatible with the standard metric.
The proof of 2) is by contradiction. Assume that a local maximum $x_0$ of $f|\bar{\Omega}$ is in the interior of a face $A$ of dimension greater than one.  Then there is a $2$-dimensional affine plane $\Pi$ through $x_0$ contained in $A$. As in Proposition \ref{negativity} we can construct a complex structure $J$ compatible with the metric of $\RR^N$ such that $\Pi$ is a complex line w.r.t. $J$. Then the restriction of $f$ to $\Pi \cap A$  is subharmonic and has a local maximum at the interior point $x_0$. Thus $f$ must be constant on the affine plane $\Pi$. But this is not possible since $f(x)$ goes to infinity as $x \in \Pi$ goes to infinity.
\end{proof}

\section{Symmetries give rise to local minima}

Assume that $G \leq \mathbb{O}(N) := \{ A \in \Mat (N,N, \RR) | A^T = A^{-1} \}$ is a finite subgroup and that $\RR^N$ is an irreducible $G$-representation.

We choose a set of points  $\Sigma = \{ w_1, \ldots , w_r \} \subset \RR^N \setminus \{0\}$ which is a union of $G$-orbits. Hence the two functions
\begin{itemize}
\item $F(x):= \prod_{i=1}^r |x-w_i|^2 = \prod_{i=1}^n F_i$, respectively
\item $f(x):= \sum_{i=1}^r \log |x-w_i|^2 = \sum_{i=1}^n f_i$.
\end{itemize}
are $G$-invariant functions.

Since the origin $0 \in \RR^N$ is a fixed point of $G$, and since $f$ is $G$-invariant, it follows that $Df(0)$ is also $G$-invariant. By the irreducibility of $(\RR^N)^{\vee} \cong \RR^N$ as $G$-representation, it follows that $Df(0) = 0$, i.e., $0$ is a critical point of $f$.

Under the above assumptions we can now prove the following:
\begin{prop} \label{sym}
Let $G \leq \mathbb{O}(N)$ be a finite subgroup such that $\RR^N$ is an irreducible $G$-representation. Suppose that $\Sigma = \{ w_1, \ldots , w_r \} \subset \RR^N \setminus \{0\}$  is a union of $G$-orbits and that $\Sigma$ is not contained in an affine plane. Then $0$ is a local minimum of $F(x):= \prod_{i=1}^r |x-w_i|^2 $ resp. of $f(x):= \sum_{i=1}^r \log |x-w_i|^2$.
\end{prop}

\begin{proof}
If the Hessian $H_{f,0}$ of $f$ in $0$ is  identically zero, then by remark \ref{levi} the Levi form $\mathcal{L}_{f,0}$ is identically zero. Lemma \ref{leviform} implies then that $\Sigma$ is contained in an affine plane, a contradiction. Therefore $H_{f,0}$ is not identically zero, whence it  is non-degenerate, since otherwise $\ker H_{f,0}$ is a non-trivial $G$-invariant subspace of $\RR^N$, contradicting the irreducibility of the $G$-representation.

We know by proposition \ref{negativity} that the positivity of the Hessian $H_f$ at a critical point is at least $N-1$.  This means that either
\begin{itemize}
\item[i)] $H:=H_{f,0} > 0$, or
\item[ii)] the positivity of $H$ is $N-1$ and the negativity is $1$.
\end{itemize}
In the second case, there are Euclidean coordinates $(x_1, \ldots , x_{N-1},y)$ such that (up to a multiplicative constant), we have
$$
H(x_1, \ldots, x_{N-1},y) = \sum_{i=1}^{N-1} a_i x_i^2 -y^2, \ a_i >0.
$$
Then $G$ leaves the cone
$$
 \mathcal{C}:=\{ x \in \RR^N | H(x) = 0 \} =  \{ (x_1, \ldots, x_{N-1},y) \in \RR^N | y^2 = \sum_{i=1}^{N-1} a_i x_i^2  \}
 $$
  invariant. This implies that $G$ leaves invariant the central line $L:=\{(0, \ldots , 0, y) \in \RR^N | y \in \RR \}$ of $\mathcal{C}$, contradicting the irreducibility of the $G$-representation. Therefore, we have seen that $H >0$, i.e., $f$ has a local minimum in $0$.

\end{proof}

\begin{rem}
It is not clear that for this choice of points $w_1, \ldots , w_r$ the function  $f(z):= \sum_{i=1}^r \log |z-w_i|^2$ is a local  Morse function. At any rate, $f$ is never a global Morse function, since the critical points appear as orbits of the symmetry group, hence they do not have different values.

But if $f$ has a local minimum in $0$ (as seen above), we shall see in the next section that, if we perturb the points $w_1, \ldots , w_r$ a little bit, obtaining points $w'_1, \ldots , w'_r$, we can achieve that

\begin{itemize}
\item $f(x) := \sum_{i=1}^r \log |z-w'_i|^2$ is  a global Morse function,
\item $f$ has a local and not global minimum.
\end{itemize}
\end{rem}

\section{The configuration space}

We go back to the notation defined in the introduction, see  definition \ref{generating}.

\noindent
$\sG \sL (r,N)$ is the open set in the space  $(\RR^N)^r$of $r$ (distinct) points in $\RR^N$,
 such that  the function $f(x) = \sum_{k=1}^r \log |x-w_k|^2$ is a global Morse function, i.e. we have a
 generic big lemniscate configuration $\Ga(f)$.

 \begin{prop}
 The complement  $$ \mathcal Y:= (\RR^N)^r \setminus \sG \sL (r,N)$$  is a real semi-algebraic set
 different from $(\RR^N)^r $, in particular  the   open set $\sG \sL (r,N)$ is non empty.
 \end{prop}

 \begin{proof}

 It is sufficient to consider the conditions that say that $w:= (w_1, \dots, w_r) \in \sG \sL (r,N)$.
 The condition that the points $w_j$ are pairwise distinct amount to the fact that $ w \notin \De_{i,j} : = \{ w | w_i = w_j\}$.

 \noindent
 Observe that $\De_{i,j} $ is a linear subspace of  codimension $N$.

 The condition that $F$ is a (local) Morse function is the condition that all critical points are nondegenerate.
  To this purpose, we consider as customary  the critical variety:
 $$ \sC \sR  : = \{ (x,w) |  x \in \RR^N , w \in  (\RR^N)^r \setminus \cup_{i<j} \De_{i,j} ,\frac{ \partial{F}}{\partial{x_j}}  (x,w)=0 \ \forall j=1, \dots, N. \}$$

 \noindent
 Here $F(x,w)$ is the real polynomial $\prod_1^r |x-w_j|^2$.

 Clearly $ \sC \sR $ is defined by $N$ polynomial equations, so it is a real  algebraic set, and its projection to
 $(\RR^N)^r \setminus  \De : = (\RR^N)^r \setminus \cup_{i<j} \De_{i,j}$ is proper by lemma \ref{gauss}.
 Now, the equation  $\frac{ \partial{F}}{\partial{x_j}}  (x,w)=0 \ \forall j=1, \dots, n$ is equivalent to the equation
$ \frac{ \partial{f}}{\partial{x_j}}  (x,w)=0 \  \forall j=1, \dots, n$ if $ x \notin \Sigma = \{ w_1, \dots, w_r\}$,
hence to the vanishing of the gradient $\grad_x f(x,w)$ with respect to the variable $x$ of the function $f(x,w)$.

\noindent
We change now variables setting $ -u_j : = (x-w_j)$.

Given, the function $\grad_x f(x,w)$, the derivative with respect to the variable $w_i$ of this vector valued function,
$$ \frac{\partial }{\partial{w_i}} [\grad_x f(x,w)]=   \frac{\partial }{\partial{w_i}} [ \frac{x-w_i}{|x-w_i|^2} ]$$
 equals to the derivative with respect to the variable $u_i$ of the function $\frac{u_i}{|u_i|^2}$.

But the derivative of the vector valued function $\frac{u}{|u|^2}$ is given by the matrix of the quadratic form (on tangent vectors $v$)
$\frac{1}{|u|^4} [ (u,u) (v,v) - 2 (u,v)(u,v)]$.   We restrict to the open set  $u \neq 0$, and without loss of generality,
by homogeneity, we can assume $|u|=1$ and indeed, after a change of orthonormal basis, that $u= e_1$.
Then the quadratic form becomes
$$
|v|^2 - 2(e_1,v)^2 = -v_1^2 +v_2^2 + \ldots + v_N^2,
$$
which is non degenrate.

We have therefore established that $ \sC \sR$ is smooth of codimension $N$ outside of the locus where $ x=w_i$.
However, the points $x=w_i$ are isolated critical points of $F$.

Hence the locus of  $ \sC \sR$ where the projection $\pi :  \sC \sR \ra (\RR^N)^r $ is not a submersion (being a submersion means
that the derivative is surjective) is a closed algebraic set and its image in $(\RR^N)^r $, by the Tarski-Seidenberg theorem
 (cf. \cite{Jacobson}, page 323 for an elementary proof), and by Sard's theorem, is a semialgebraic set of dimension strictly smaller than $Nr$.

Now, the key well known  fact is that the isolated critical points of the function $F$ are exactly the points of $ \sC \sR$ where
  the projection $\pi :  \sC \sR \ra (\RR^N)^r $ is  a submersion.

  The final condition  that $f$ be a global Morse function runs as follows: we have a non empty open set
  for which $F$, hence $f$, is  a local Morse function; this set is the
  complement of the above semialgebraic set, that we denote by   $ \sL (r,N)$ (observe that for an r-tuple of points in $ \sL (r,N)$ the singular level sets may  contain  more than one singular point).

  Over the open set $ \sL (r,N)$  the critical points are a finite set,
  and the condition that $f_w$ ($f_w(x) = f (x,w)$)  is  a global Morse function is that the values of $F_w$ on the critical points which are not in $\Sigma$
  are pairwise distinct.
  We are thus removing  another closed semialgebraic set.

  That $\sG \sL (r,N)$ is non empty follows  by considering points $w_1, \dots, w_r$ which lie in $\RR^2$.
  Using a complex structure where  $\RR^2 = \CC$, we reduce using corollary \ref{orthodec}  to the case
  of polynomial lemniscates on $\CC$, dealt with in \cite{catanesepaluszny}.

  \end{proof}

\begin{prop}\label{minimastability}
 Assume that $\sG \sL (r,N)$ is everywhere dense, equivalently, its  complementary set is a semialgebraic set of real dimension  $< Nr$.

Then, given  an r-tuple $w \in (\RR^N)^r$ of points $ w_1, \ldots , w_r \in \RR^N$ (yielding a set of points $\Sigma_w := \{ w_1, \ldots , w_r \}  \subset \RR^N$) such that
\begin{itemize}
\item $f_w(x) := \sum_{i=1}^r \log |z-w_i|^2$ is a local Morse function,
\item $f_w$ has  $h$ local minima in $\RR^N \setminus \Sigma_w$,
\end{itemize}
then for each $\de > 0$,  there is another r-tuple of points $w'_1, \dots, w'_r$, with $|w'_i - w_i| < \de$
such that
\begin{itemize}
\item $f_{w'}(x) := \sum_{i=1}^r \log |z-w'_i|^2$ is a global Morse function,
\item $f_{w'}$ has  $h$ local minima in $\RR^N \setminus \Sigma_{w'}$,
\end{itemize}
\end{prop}

\begin{proof}
The second assertion follows  from the first, since we can find an r-tuple $w'$ very close to $w$ and lying in $\sG \sL (r,N)$.

Then, for $\de$ sufficiently small, the difference $|f_w - f_{w'}|$ is smaller, on any given compact $K$ containing the convex hull of the set $\Sigma_w$, than any given $\epsilon > 0$,
provided $|w'_i - w_i| < \delta$.

Let now $y$ be a local minimum for $f_w$ which is not a global minimum. Then there is a constant $r$ such that the closed ball $\overline{B(y,r)}$
contains no other critical points, and, for $x \in \partial B(y,r/2)$, $ f_w(x) > f_w(y) + 2 \epsilon(y) $, where $ \epsilon(y) > 0$ is a constant.

Set $  \epsilon : = min_y   \epsilon(y)$ and choose $\de$ as above: then $f_{w'}$ still possesses a local minimum inside $\overline{B(y,r/2)}$.

\end{proof}

\begin{prop}
$\sG \sL (r,N)$ is everywhere dense.
\end{prop}

\begin{proof}
Assume the contrary: then there is a connected component $U$ of $ \sL (r,N)$ which has empty intersection with $\sG \sL (r,N)$
and such that it has a common boundary $M$ with a connected component $U'$ of $\sG \sL (r,N)$.

Take a general point $w$ of the common boundary $M$, and take an analytic arc $I$ transversal to $M$ at $w$;  apply the following two lemmas \ref{gen1}
and \ref{gen2}.  Then the inverse image $J$ of $I$ inside the critical variety $ \sC \sR$ consists of several arcs $J_h$ which map homeomorphically
to $I$, plus there is (possibly) another arc $J'$ which maps with degree 2 to the part of $I$ lying in one of the two domains $U$, respectively $U'$.

Set $I_U : = I \cap U$, and similarly $I_{U'} : = I \cap U'$. By the hypothesis, there are two arcs $A,B$ of $I_U$ on which the critical values are the same; by analytic continuation,
these arcs $A,B$ cannot  both respectively be part of two arcs of the form $J_h$. Hence one, say $A$,  of them lies in the arc $J'$. By lemma \ref{gen2} $B$ cannot lie in an arc
of the form $J_h$, otherwise, by analytic continuation, the critical values on $J_h \cap I_{U'} $ would be imaginary.
Finally, the possibility that both arcs $A,B$  lie in $J'$ contradicts lemma \ref{gen2} .

\end{proof}

\begin{lemma}\label{gen1}
Assume that we have a degenerate critical point of the function $f = f_w$,
where there are local coordinates  $(u_1, \dots, u_N)$ such that
$$ f (u) = u_1^2 + \dots + u_{N-1}^2 + c + \eta u_N^3,$$ where $ \eta = \pm 1$.

Then for   $w'$ in a neighbourhood of $w$ we have a deformation of $f$ of the form
$$ f_{w'} (u) = u_1^2 + \dots + u_{N-1}^2 + c' + \eta u_N^3 + t(w') u_N,$$
for $t(w')$ an appropriate analytic function of $w'$.

The critical point $u=0$ deforms to two real nondegenerate critical points
$$ u_1 = u_2= \dots =  u_{N-1} = 0, \sqrt{3} u_N = \sqrt{- \eta t(w')}$$
for the points $w'$ where the function $\eta t(w')$ is negative,
and to two  imaginary critical points for the points $w'$ where the function $\eta t(w')$ is positive.

In particular, the critical values of  the function $f_{w'} $ at the two real critical points are distinct as soon as the points are distinct (i.e., for $t(w') \neq 0$); while the critical values at the imaginary critical points are non real.

\end{lemma}

\begin{lemma}\label{gen2}
The locus of r-tuples $w$ such that $f_w$ has at least two degenerate singular points (counted with multiplicity) has codimension at least $2$.
\end{lemma}
We shall provide the proof of the above lemmas in another paper.

\begin{rem}
In the next section we shall give explicit examples, first  of the situation in Proposition \ref{sym}, then of cases where one has many local (non global) minima,
 and we show that in these examples $f$ is in fact a local Morse function (but never a global Morse function).
\end{rem}

\section{Elementary examples}
In this section we give some explicit examples of points $\{ w_1, \ldots , w_r \}  \subset \RR^N$, such that $f(x) := \sum_{i=1}^r \log |z-w_i|^2$ has one or more local minima in $\RR^N \setminus \Sigma$.

We omit most of the  elementary calculations, which can be found in the arXiv version of the paper, \cite{3Darxiv}.

\subsection{The hypercube in $\RR^N$}
Let $N \geq 3$ be a natural number and consider the midpoints of the big faces of the hypercube,
i.e., the points $P_1, P_2, \ldots , P_{2N} \in \RR^{N}$, with
$$
P_i:= e_i, \ P_{N+i} := -e_i, \ 1 \leq i \leq N.
$$
\noindent
Here $e_i$ is the i-th standard basis vector of $\RR^N$. Let
$$
F(x):= \prod_{i=1}^N |x-e_i|^2|x+e_i|^2.
$$

\begin{prop}
Then $F$ has $2N$ absolute minima in the points $x = \pm e_i$, $1=1, \ldots, n$, a local minimum in $x=0$ and $2N$ non degenerate critical points of negativity 1 in $x = \pm \sqrt{\frac{N-2}{N}} e_i$.
\end{prop}

In fact,
\begin{multline}
F(x) = F(x_1, \ldots, x_N) = ((x_1-1)^2 + x_2^2 + \ldots + x_N^2) ((x_1+1)^2 + x_2^2 + \ldots + x_N^2) \cdot \\
\cdot (x_1^2 + (x_2-1)^2 + \ldots + x_N^2) (x_1^2 + (x_2+1)^2 + \ldots + x_N^2) \cdot \ldots \cdot \\
\cdot(x_1^2 + x_2^2 + \ldots + (x_N-1)^2) (x_1^2 + x_2^2 + \ldots +(x_N-1)^2) = \\
= (x_1^2 -2x_1+1 +x_2^2 + \ldots + x_N^2) (x_1^2+2x_1+1 + x_2^2 + \ldots + x_N^2) \cdot \ldots\\
\cdot \ldots \cdot (x_1^2 + x_2^2 + \ldots + x_N^2-2x_N +1) (x_1^2 + x_2^2 + \ldots +x_N^2+2x_N +1)= \\
= 1 + 2(N-2)(x_1^2 + x_2^2 + \ldots + x_N^2) + f_{\geq 3}(x).
\end{multline}
Here $ f_{\geq 3}(x) \in \RR[x_1, \ldots , x_N]$ is a sum of monomials of order $\geq 3$.
This implies that for $N \geq 3$,  $F$ has a local minimum in $x=0$.

To find the   Hesse matrix of $f = \log(F)$ at the critical points $\pm \lambda e_i$, $\lambda = \sqrt{\frac{N-2}{N}}$,
 in view of  the symmetry, it is enough to find it at the point $x = \lambda e_1$.

One finds that $H_{f,x}$ at the point $x = \lambda e_1$ is diagonal, and that

\[\begin{aligned}
(H_{f,\lambda e_1})_{11} =  -2N^2\frac{N-2}{N-1}.\\
\end{aligned}\]

while for  $1 < i$
\[(H_{f,\lambda e_1})_{ii} = 2N^2 \left(\frac{(N^2 - 2N - 2)}{2N(N-2)^2} + \frac{N-1}{N} \right) .\]

Hence  \[ H_{f,\lambda e_1} \sim \mathrm{diag}(-2N^2, 2N^2, \cdots, 2N^2)\]
as $N \to \infty$.

\begin{remarkk}
We have later found that this example, for $N=3$,i.e. in the case of the octahedron,
had already been given in section $5$ of \cite{motzkin-walsh}.
\end{remarkk}
\subsection{Three elementary  examples in $\RR^3$} \

\noindent
1) Consider the four vertices  $w_1=e_1$, $w_2=e_2$, $w_3=e_3$, $w_4=e_1+e_2+e_3$ of the regular symplex  in $\RR^3$.

Here
$$
F(x):= \prod_{i=1}^4 |x-w_i|^2.
$$
 has  four absolute minima in $w_1, w_2, w_3, w_4$, a local minimum in the barycenter $B:= \frac 12 w_4$ and 4 non degenerate critical points (of negativity 1) in
$$
y_i : =  \frac 13 w_i + \frac 23 B , \ 1 \leq i \leq 4.
$$

\medskip
\noindent
2) Consider the following eight vertices of the regular cube  in $\RR^3$:
$$
\{ w_1, \ldots , w_8 \}= \{ 0,e_1, e_2, e_3, e_1+e_2,e_1+e_3, e_2+e_3, e_1+e_2+e_3\}.
$$
$$
F(x):= \prod_{i=1}^8 |x-w_i|^2.
$$

Then $F$ has eight absolute minima in $w_1, \ldots, w_8$, a local minimum in  $\frac 12 w_8$ and further 8 non degenerate critical points (of negativity 1).

\medskip
\noindent
3) Consider the following six points  in $\RR^3$:
$$
\{w_1, \ldots , w_6\} = \{e_1, e_2, e_3, e_1+e_2,e_1+e_3, e_2+e_3\}.
$$
$$
F(x):= \prod_{i=1}^6 |x-w_i|^2.
$$

Then $F$ has six absolute minima in $w_1, \ldots, w_6$, a local minimum in $\frac 12 (e_1+e_2+e_3)$ and further 6 non degenerate critical points (of negativity 1).

\subsection{The vertices of the regular simplex in $\RR^N$} \

\noindent
 Let $N \geq 4$ be a natural number and consider $w_1, w_2, \ldots , w_{N} \in \RR^{N}$, where
$$
w_i:= e_i, \ 1 \leq i \leq N.
$$
\noindent
Here $e_i$ is the i-th standard basis vector of $\RR^N$. Let
$$
F(x):= \prod_{i=1}^N |x-e_i|^2
$$
and $f(x) := \log(F(x))$.

\begin{prop}\label{hminima}
Then $F$ has $N$ absolute minima in the points $x = e_i$, $i=1, \ldots, N$,
a local minimum in the barycentre $x=B:= \frac{1}{N}\sum_{i=1}^N e_i$ and $N$ non degenerate critical points
of negativity 1 in $Q_i := \frac{2}{N-1} B + \frac{N-3}{N-1} e_i$, $i=1, \ldots, N$.

More precisely,  $H_{f,B}$ has two eigenvalues:
\begin{itemize}
\item $\frac{2N^3}{N(N-1)}$ with multiplicity 1 whose eigenvector is parallel to $B$
and
\item $2(N-3)(\frac{N}{N-1})^2$ with multiplicity $N-1$.
\end{itemize}
The Hessian matrix $H_{f,Q_i}$ has three eigenvalues:
\begin{itemize}
\item $-\frac{(N-3)N^2}{2(N-2)}$ has multiplicity one with eigenvector $B-e_i$,
\item $\frac{(N-1)N^2}{2(N-2)}$ with multiplicity one with eigenvector $B$, and
\item $\frac{8+(-3+N) N \left(4+N^2\right)}{2 (-2+N)^2}$ with multiplicity $(N-2)$.
\end{itemize}
\end{prop}

\subsection{The regular triangular prism: an example with two local minima} \
\noindent

This new example is a warm up for the more complicated ones which shall be described in the next section.

Fix $a \in \RR_+$ and consider the following six points: $w_j:=(u_j,a), \ w'_j:=(u_j,-a) \in \CC \times \RR = \RR^3$, where  for $1 \leq j \leq 3$ we set $u_j := e^{j \frac{ 2\pi i}{3}}$. Set
$$
F_a(x):= \prod_{j=1}^3 |x-w_j|^2 \cdot  |x-w'_j|^2.
$$
$F_a$ has six absolute minima in the points $w_j, w'_j$.

We shall prove the following
\begin{prop} \label{2minima} \
\begin{enumerate}
\item \underline{$a=1$:} $F_1$ has a critical point in $0$, whose Hessian has nullity 1 (and positivity two).
\item  \underline{$a<1$:} $F_a$ has a local minimum in $0$ with non degenerate Hessian and no further critical points of the form $(0,0,x_3)$.
\item \underline{$a>1$:} $F_a$ has a non degenerate critical point in $0$ with negativity 1, and two local minima in $(0,0, \pm \sqrt{a^2-1})$ .
\end{enumerate}
\end{prop}

\section{$3h$ points on an equilateral triangular prism,  with $h$ preassigned local (non absolute) minima}\label{preassigned}

Let $r_1 < r_2 < \cdots < r_h$ be arbitrary real numbers, $h>1$.
We are going to construct $3h$ points  $w_1,w_2, \cdots,w_{3h} \in \mathbb{R}^3$ such that $F(x) = \prod_{j=1}^{3h} |x - w_j|^2$
has $h$ local (non absolute and non degenerate) minima at the points $(0,0,r_j)$, $j = 1,\cdots , h$.\\

Actually, $F$ is going to have also $h-1$ saddles (non degenerate critical points of negativity 1) on the $x_3$-axis.

\subsection{The auxiliary polynomial $P$}

Given the $r_j$'s take $s_j$ such that $r_j < s_j < r_{j+1}$ for  $j=1,\cdots,h-1$.
Let $P(X)$  be the polynomial
$$ P'(X) = \left( \prod_{j=1}^{h-1}(X - r_j)(X-s_j) \right) (X-r_h) \, .$$

Then $\mathrm{deg}(P) =2h $ and $P(X)$ is bounded from below. So we can assume w.l.o.g. that $P(X) > 0$ for all $X \in \mathbb{R}$.\\
By construction $P(X)$ has $h$ local minima at the $r_j$'s and $h-1$ local maxima at the $s_j$'s.\\

Decompose $P(X)$ as \[ P(X) = \prod_{j=1}^h P_j(X) \, \]
where $P_j(X)$ are degree two monic real polynomials without real roots.
Observe that $P_j \neq P_k$ for $j \neq k$ otherwise the derivative would have a double root,
contradicting our construction of  $P'(X)$.

We can also assume w.l.o.g.  that $P_j(X) > 0$ for all $  X  \in \RR$.

Hence there are real numbers $a_j, b_j \in \mathbb{R}$ such that
\[ P_j(X) = (X - a_j)^2 + b_j^2\]
and we can assume that $b_j > 0$.\\

Summing up we have:

 \begin{equation}\label{Ache} P(X) = \prod_{j=1}^h \left( (X - a_j)^2 + b_j^2 \right) \end{equation}

\subsection{The $3h$ points $w_1, \dots, w_{3h}$} \

 We regard now $\mathbb{R}^3 $ as $ \mathbb{C} \times \mathbb{R}$.
For each $j \in \{1,\cdots,h\}$ we consider the following  3 points $w_j^1, w_j^2, w_j^3$ defined as follows:

\[ w_j^i := (\xi^i b_j , a_j) \]
where $\xi = e^{\sqrt{-1} \frac{2\pi}{3}}$ and the $a_j,b_j$'s  are as in the factorization of $P(X)$ given  in equation (\ref{Ache}).
\begin{prop}
 Let $F(x) := \prod_{j=1}^h |x - w_j^1|^2 |x - w_j^2|^2 |x - w_j^3|^2 $. Then:
 \begin{enumerate}
 \item $F$ has $h$ local (non absolute) nondegenerate minima in  $(0,r_j) \in  \CC \times \RR$, $1 \leq j \leq h$;
 \item $F$ has $h-1$ saddle points (nondegenerate critical points of negativity 1)  in the points $(0,s_j)$, $1 \leq j \leq h-1$;
 \item $F$ has $3h$ absolute minima in the points $w_j^i$, $1 \leq j \leq h$, $1 \leq i \leq 3$;
 \end{enumerate}
\end{prop}

\begin{proof} \

\noindent
{\bf Step I:} {\it The $h + h-1$ points $(0,r_j)$ and $(0,s_j)$ are critical points of $F$.\rm}

To see this write $x = (z,t)$ so that
\[F(x) = F(z,t) = \prod_{j=1}^h \left(|z - \xi b_j|^2 + (t-a_j)^2 \right) \left(|z - \xi^2 b_j|^2 + (t-a_j)^2 \right) \left(|z - b_j|^2 + (t-a_j)^2 \right)\]
hence
 \[ F(\xi z, t) = F(z,t).
  \]
I.e.,  $F$ is invariant for  the 120 degree rotation on the first factor $\mathbb{C}$.

This implies that the gradient $\nabla F$ at the points $(0,r_j)$ and $(0,s_j)$ has zero $\mathbb{C}$-component.
The $\mathbb{R}$-component of $\nabla F$ is then $\frac{\mathrm{d} F(0,t)}{\mathrm{d} t}$. Now
\[ F(0,t) = \prod_{j=1}^h \left(b_j^2 + (t-a_j)^2 \right) \left(b_j^2 + (t-a_j)^2 \right) \left(b_j^2 + (t-a_j)^2 \right)= P(t)^3\]
where $P$ is as in equation (\ref{Ache}). So, as we claimed, $\frac{\mathrm{d} F(0,t)}{\mathrm{d} t}$ is zero at the $r_j$'s and $s_j'$s.\\

\noindent
{\bf Step II:} {\it At each critical point of the form $(0,r_j)$ and $(0,s_j)$ the Hessian matrix $H_{F}$ has both factors $\mathbb{C}$ and $\mathbb{R}$
as invariant subspaces. Moreover $H_{F}|_{\mathbb{C}} = \lambda \mathrm{Id}_{\mathbb{C}}$, for $\lambda \in \mathbb{R}$. \rm}
Let in fact $R_{120}:\mathbb{R}^3 \to \mathbb{R}^3$ be the rotation of 120 degrees, induced by the multiplication by $\xi$  on the first factor $\mathbb{C}$. Then, as observed above, $F \circ R_{120} = F$. The critical points $(0,r_j)$ and $(0,s_j)$ are fixed by $R_{120}$ hence: \[ R_{120} H_{F,\mathbf{p}} = H_{F,\mathbf{p}} R_{120} \, \]
where $\mathbf{p} \in \{ (0,r_j), (0,s_j) \}$. Since the $\mathbb{R}$ factor is the unique fixed line by $ R_{120}$ it follows that $H_{F,\mathbf{p}}$ preserves the $\mathbb{R}$ line hence $H_{F,\mathbf{p}}$ also preserves $\mathbb{C}$. The $\mathbb{Z}_3$-action generated by $ R_{120}$ is irreducible on $\mathbb{C}$. It follows that $H_{F}|_{\mathbb{C}} = \lambda \mathrm{Id}_{\mathbb{C}}$ as we claimed.\\

\noindent
{\bf Step III:} {\it The points $(0,r_j)$ are local (non absolute) minima of $F$ whilst $H_{F}$ has negativity 1 and is nondegenerate at the points $(0,s_j)$. \rm }
Since $h > 1$ the $3h$ points are not coplanar. So at each critical point $\mathbf{p} \in \{ (0,r_j), (0,s_j) \}$ the constant $\lambda$ in the first block of the Hessian matrix $H_{F,\mathbf{p}}$ must be positive, i.e. $\lambda > 0$. The constant in the $\mathbb{R}$ direction is given by
\[ \frac{\mathrm{d}^2 F(0,t)}{\mathrm{d} t^2} = \frac{\mathrm{d}^2 ( P(t)^3 )}{\mathrm{d} t^2} \]

By construction $P(X)$ has non degenerate local minima at the $r_j$'s and non degenerate local maxima at the $s_j$'s.
Since $P > 0$ the cubic power does not change the sign of the derivatives and  we get that also $F(0,t)$ has has non degenerate local minima at the $r_j$'s and non degenerate local maxima at the $s_j$'s.
\end{proof}

\begin{rem}
One can show that, for general choice of the numbers $r_j, s_j$,
 $F$ has exactly $3 h$ further saddle points (nondegenerate critical points of negativity 1).

\end{rem}

\begin{figure}[h!]
\centering
\includegraphics[trim=5cm 8cm 3cm 9cm, clip,draft=false,width=8cm,height=7cm]{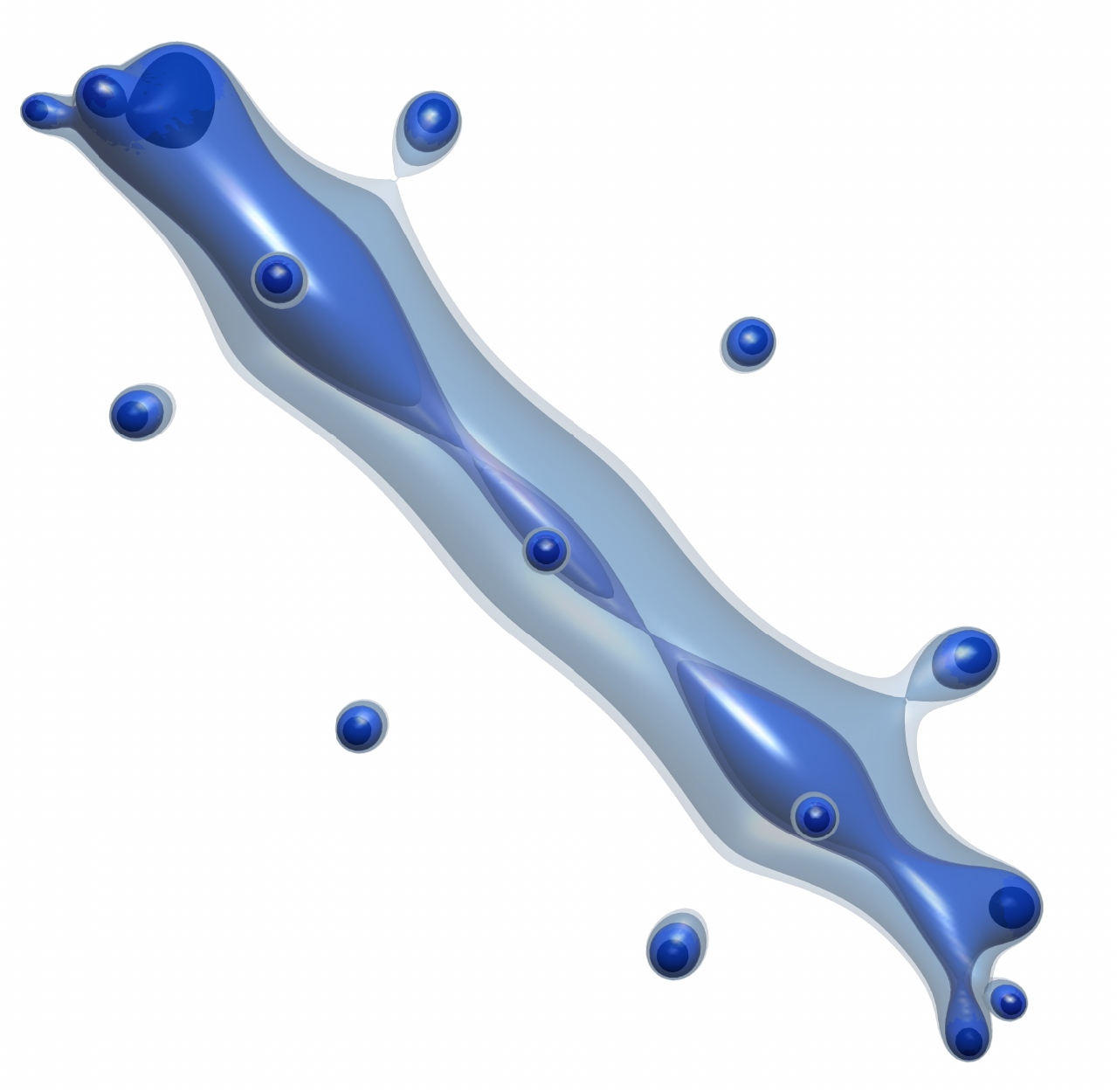}
\caption{Several lemniscates of a perturbation of the configuration for $r=15$ constructed with the method of this section.}
\end{figure}

\clearpage
\section{Appendix}

Let $w_1,\cdots,w_r$  be $r$ points of $\mathbb{R}^{n}$
and let $f(x) : = \sum_{i = 1}^r  \log(\| x - w_i \|^2)$ be the associated logarithmic potential.

In this appendix we prove the following theorems.

\begin{theo}\label{MorseCairns} The critical points of the logarithmic potential $f$ are isolated.
That is to say, there is no arc of critical points.
\end{theo}

\begin{theo}\label{spheres} A smooth connected component of a level set $f^{-1}(c) \subset \mathbb{R}^{n}$ is diffeomorphic to the standard unit sphere $\mathbb{S}^{n-1} \subset \mathbb{R}^n $.
\end{theo}

We notice that the first theorem is a consequence of the second. Indeed, if the critical points are not isolated then by Theorem \ref{index}
there is a circle of   local minima. Then by Lemma \ref{localform} the level sets of $f$ near the circle are diffeomorphic to  an orientable fibre bundle over $\mathbb{S}^1$  with fibre $ \mathbb{S}^{n-2}$, which has first Betti number $b_1 = 1$ for
$n \geq 4$, and $b_1=2$ for $n=3$, since for $n=3$ an orientable such bundle is diffeomorphic to a torus.

\begin{rem} The question about existence of an arc of critical points for the electrostatic potential generated by a finite number of charges of the same sign was posed by Cairns and Moser in their book \cite[page 294]{cairns-morse}. It is still an open problem.
\end{rem}

 \noindent{\it Proof of Theorem \ref{spheres}.}
 
We want  to show that all the smooth connected components of the level sets $f^{-1}(c) \subset \mathbb{R}^{n}$ 
 are  diffeomorphic to the standard unit sphere $\mathbb{S}^{n-1} \subset \mathbb{R}^n $.
 
 Now, all the  critical points of $f$ have negativity at most $1$, and are either non degenerate, or belong 
 to a circle of minima.  These circles of minima yield,
 as we already saw,  connected components with first Betti number $ b_1 \geq 2$.
 
 Now, by Morse theory, passing through a non degenerate critical point with negativity equal to one,
 either the zeroth Betti number $b_0$ diminishes by one, and this means that we have two connected components of which
 we take the connected sum, or the first Betti number $b_1$ grows by $1$.
 
 This shows that, once the first Betti number $b_1 \geq 1$, it shall never go back to be equal to zero.
 
 It suffices therefore to show:
 
 \begin{lemma}\label{bigspheres}
 For $ c >> 0$ the level sets $f^{-1}(c) $ are spheres.
 
 \end{lemma}
 
 Because, then the Betti number of $f^{-1}(c) $ is $b_1 = 0$ for  $ c >> 0$ $\Rightarrow $ the first Betti number
 of each smooth component of a level set is always zero, hence these are just connected sums of  spheres, i.e., spheres.
 
 {\it Proof of Lemma \ref{bigspheres}}
 
 Consider the function
 $$F(x) : = \Pi_{i = 1}^r  (\| x - w_i \|^2)$$
 and assume that the barycentre $\sum_{i = 1}^r   w_i  = 0.$
 
 Then $ F(x) = \| x  \|^{2r} + g(x)$, and for $\| x \| \geq 1$ we have $ | g(x) | \leq  C \| x  \|^{2r-2}$, where $C$ is a fixed constant. 
 
 For $c >>0$ we get a smooth hypersurface
 $$X_c : = \{ x | F(x) = c\} = \{ x | \| x  \|^{2r} + g(x)=c\}.$$
 
 Assume that $y$ is a point in the unit sphere, i.e., $\| y\| = 1$. Then we claim that  the half line $\RR_+ y$ intersects
 $X_c$ in precisely one point.
 
 In fact, this point corresponds to the positive value of $ t \in \RR$, such that 
$$  \| t  \|^{2r} + g(t y )= c .$$
We have  $|g(t y )|  \leq C t^{2r-2}$ and $| \frac{d}{dt} g(ty)| \leq C' |t|^{2r-3}$  for $|t| \geq 1$, where we may assume $ C' \geq C$.

Therefore there is $R > 1 $ such that  the function $F( ty)$ is strictly monotone growing for $t > R$, and if we let $ c >  (C + 1)  R^{2r}$,
then $t$ is uniquely determined.

We have then given a diffeomorphism between the unit sphere and $X_c$ for $c > >0$.

\medskip

  {\bf QED for  the   proof  of Lemma \ref{bigspheres} and of Theorem \ref{spheres}}

\bigskip

  We shall now give  another    proof  of Theorem \ref{MorseCairns}.

\noindent \it Second  Proof of Theorem \ref{MorseCairns} \rm Without loss of generality we can assume that $n=2m$.
The proof is based on the important observation that the logarithmic potential $f$ is strictly plurisubharmonic with respect to {\bf any} complex structure of $\mathbb{R}^{2m}$ compatible with the standard flat Riemannian metric, see Lemma \ref{leviform}. So we are going to take advantage of the freedom in the choice of a convenient complex structure.
For example, if the points $w_1,\cdots,w_r$ are contained in a real 2-plane $\Pi$ then by changing the complex structure of $\mathbb{R}^{2m}$
we can assume that $\Pi$ is a complex line. Then the critical points are the zeros of the derivative of the polynomial $P(z) = \Pi_{i=1}^r (z - w_i)$, where $z$ is a complex coordinate on $\Pi$, hence the critical points are isolated.\\

So, from now on, we shall assume that  the points $w_1,\cdots,w_r$ are not contained in a $2$-plane and that there are non isolated critical points.
Again, by Theorem \ref{index},
there is a circle $\Gamma$ of critical points of $f$, which are  local minima of $f$.
We are going to derive  a contradiction.

Observe that  $\Gamma \subset f^{-1}(c), c \in \mathbb{R}$, is a connected component of a fibre $f^{-1}(c)$ of $f$, 
since the positivity of the Hessian is at least $2m-1$, see Theorem \ref{index}.

By using Jesse Douglas' solution of the  Plateau Problem   \cite[page 269]{Douglas}
we get a harmonic map $\varphi : \overline{D} \to \mathbb{R}^{2m}$  where $\overline{D}$ is the closed disc in $\RR^2$, and
with the property   that  $\Gamma = \varphi(\partial D)$.
As explained by Douglas, $\varphi$ is conformal except at isolated points $p \in D$ where $\mathrm{d} \varphi_ p = 0$.

 The composition $h = f \circ \varphi$ must take its maximum value at a point $ x_0 $ in the interior of $ \overline{D}$.  Otherwise, i.e. if $x_0$ is in the boundary of the disk,  then $h$ is constant hence $\varphi(\overline{D}) \subset f^{-1}(c)$ which contradicts
the fact that  $\Gamma$ is  a connected component of $f^{-1}(c)$ according to Theorem \ref{index} (observe that  $\varphi(\overline{D})$ is connected,  and it contains $\Gamma$ strictly since it is contractible). \\

The theorem is consequence of the following claim.\\

\noindent
{\bf Claim: the function $h$ is subharmonic}

Let $x \in D$ be a general point,  i.e. at $x$ the harmonic map $\varphi$ is conformal.
Then $ \Pi := \mathrm{d} \varphi_ x (T_x D) \subset T_{\varphi(x)} \mathbb{R}^{2m}$ is a 2-plane.
Let $J$ be a complex structure of $\mathbb{R}^{2m}$ compatible with the standard metric such that $J(\Pi) = \Pi$.

The Laplacian  $\Delta h$ at $x$ (\cite[Corollary 3.3.13,page 76]{BairdWood}) satisfies:
\[ \Delta h = \mathrm{d}f    (\tau (\varphi))   + H_{f,x} (\mathrm{d} \varphi (e_1), \mathrm{d} \varphi (e_1) ) + H_{f,x} (\mathrm{d} \varphi (e_2),\mathrm{d} \varphi (e_2) ) ,\]

where $e_1,e_2$ is an orthonormal frame at $x$ and $\tau$ is the tension field of $\varphi$.

Since $\tau \equiv 0$ because $\varphi$ is harmonic (\cite[Theorem 3.3.3, page 73]{BairdWood})we get
\[ \Delta h = H_{f,x} (\mathrm{d} \varphi(e_1), \mathrm{d} \varphi (e_1) ) + H_{f,x} (\mathrm{d} \varphi (e_2),\mathrm{d} \varphi (e_2) ).\]

Now,   using that $J \mathrm{d} \varphi (e_1)  = \mathrm{d} \varphi(e_2)$ (since $\Pi$ is $J$ invariant and $\varphi$ is conformal) the vector
\[ Z := \mathrm{d} \varphi (e_1) + J \mathrm{d} \varphi (e_1) \, \mathrm{i} \] is of type $(1,0)$ w.r.t. $J$. Thus

\[ \begin{aligned} \Delta h &= H_{f,x} (\mathrm{d} \varphi (e_1), \mathrm{d} \varphi(e_1) ) + H_{f,x} (\mathrm{d} \varphi (e_2),\mathrm{d} \varphi (e_2) ) = \\ &= H_{f,x}(Z,\overline{Z}) = 2 {\mathcal{L}}_{f,x}(Z,\overline{Z}) > 0 , \end{aligned}\]

where ${\mathcal{L}}_{f,x}$ is the Levi form of our logarithmic potential $f$ which is strictly plurisubharmonic, hence the last inequality.\\
Since $x$ is general we get that \[ \Delta h \geq 0 \]
on the whole disk by continuity of the Laplacian of $h$.
 {\bf QED of the claim}.

 The theorem follows since the subharmonic function $h$ cannot have a maximum in the interior of the disk.

 {\bf QED for  the second proof of Theorem \ref{MorseCairns} }

\section{Acknowledgements}

We would like to thank Umberto Zannier for providing us with a useful lemma  which was used in our first  argument to show that the number
of local non absolute minima can tend to infinity. A.J. Di Scala would like to thank Daniel Peralta-Salas and Martin Sombra for the discussions that led to the statement and proof of Theorem \ref{spheres}.

\bigskip
\noindent
{\bf Author's Adresses:}

\noindent
Ingrid Bauer, Fabrizio Catanese\\
Mathematisches Institut der Universit\"at Bayreuth, NW II\\
Universit\"atsstr. 30;
D-95447 Bayreuth, Germany\\

\noindent
Antonio J. Di Scala\\
Dipartimento di Scienze Matematiche `G.L. Lagrange'\\
Politecnico di Torino, Corso Duca degli Abruzzi 24;
10129, Torino, Italy\\

\bigskip

\end{document}